\documentclass[11pt,a4paper]{amsart}
\usepackage{mathrsfs}
\usepackage{syntonly}
\usepackage{amsmath}
\usepackage{amsthm}
\usepackage{amsfonts}
\usepackage{amssymb}
\usepackage{latexsym}
\usepackage{amscd,amssymb,amsopn,amsmath,amsthm,graphics,amsfonts,mathrsfs,accents,enumerate,verbatim,calc}
\usepackage[dvips]{graphicx}
\usepackage[colorlinks=true,linkcolor=red,citecolor=blue]{hyperref}
\usepackage[all]{xy}
\usepackage{enumitem}

\date{}
\allowdisplaybreaks[4] \footskip=15pt
\renewcommand{\uppercasenonmath}[1]{}

\numberwithin{equation}{section} \theoremstyle{plain}
\newtheorem{lem}{Lemma}[section]
\newtheorem{cor}[lem]{Corollary}
\newtheorem{prop}[lem]{Proposition}
\newtheorem{thm}[lem]{Theorem}

\newtheorem{definition}[lem]{Definition}
\newtheorem{Ex}[lem]{Example}
\newtheorem{Quest}[lem]{Question}
\newtheorem{Property}[lem]{Property}
\newtheorem{Properties}[lem]{Properties}
\newtheorem{Subprops}{}[lem]
\newtheorem{Para}[lem]{}

\newtheorem{fact}[lem]{Fact}

\newtheorem{rem}[lem]{Remark}

\newenvironment{ex}{\begin{Ex}\rm}{\end{Ex}}

\newtheorem*{ack*}{ACKNOWLEDGEMENTS}

\newcommand{\pf}{\noindent\begin {proof}}
\newcommand{\epf}{\end{proof}}

\newcommand{\D}{{\mathcal{D}}}

\newcommand{\GPC}{\mathcal{GP}(\mathcal{C})}
\newcommand{\GPD}{\mathcal{GP}(\mathcal{D})}

\newcommand{\C}{\mathcal{C}}

\newcommand{\s}{\stackrel}

\newcommand{\RM}{{_{R}\mathcal{M}}}

\newcommand{\SM}{{_{S}\mathcal{M}}}

\newcommand{\Ext}{{\rm Ext}}

\newcommand{\Hom}{{\rm Hom}}

\newcommand{\ra}{\rightarrow}

\begin{document}
\begin{center}
{\Large  \bf  Frobenius functors and Gorenstein projective precovers}

\vspace{0.5cm}  Jiangsheng Hu, Huanhuan Li, Jiafeng L${\rm\ddot{u}}$ and Dongdong Zhang\footnote{Corresponding author.\\
\indent Jiangsheng Hu was supported by the NSF of China (Grants Nos. 11671069, 11771212), Qing Lan Project of Jiangsu
Province and Jiangsu Government Scholarship for Overseas Studies (JS-2019-328). Jiafeng L\"{u} was supported by NSF of China (Grant No.11571316) and Natural
Science Foundation of Zhejiang Province (Grant No. LY16A010003).
} \\
\medskip

\end{center}

\bigskip
\centerline { \bf  Abstract}
\medskip

\leftskip10truemm \rightskip10truemm \noindent We establish relations between Gorenstein projective precovers linked by Frobenius functors. This is motivated by an open problem that how to find general classes of rings for which modules have Gorenstein projective precovers. It is shown that if $F:\C\rightarrow\D$ is a separable Frobenius functor between abelian categories with enough projective objects, then every object in $\C$ has a Gorenstein projective precover provided that every object in $\D$ has a Gorenstein projective precover. This result is applied to separable Frobenius extensions and excellent extensions.
\leftskip10truemm \rightskip10truemm \noindent
\hspace{1em} \\[2mm]

\leftskip0truemm \rightskip0truemm
%\bigskip
\section { \bf Introduction}
%\bigskip

Gorenstein homological algebra is the relative version of homological algebra
that replaces the classical projective (injective, flat) objects with the
Gorenstein projective (Gorenstein injective, Gorenstein flat) ones.
A basic problem in Gorenstein homological algebra is to try to get Gorenstein analogues of results in the classical homological algebra. Perhaps the fundamental problem is to find general classes of rings for which modules have Gorenstein projective precovers. So
far the existence of Gorenstein projective precovers (of left modules) is known over a left coherent ring for which the projective dimension of any flat right module is finite (see \cite{EIO}). Examples of such rings include but are not limited to: Gorenstein rings, commutative noetherian rings of finite Krull dimension, as well as two sided noetherian rings $R$ such that the injective dimension of $R$ (as a right $R$-module) is finite. But
for arbitrary rings this is still an open question. Work on this problem can be seen in \cite{Asadollahi,BHG,EEI,EIO,EIK,Jorgensen} for instance.

Recall that a pair of functors $(F,G)$ is said to be a Frobenius pair \cite{casta} if $G$ is at the same time a left and a right adjoint of $F$. That is a standard name which we use instead of Morita's original ``strongly adjoint pairs" in \cite{Morita}. The functors $F$ and $G$ are known as Frobenius functors.
A prominent example of Frobenius functors is provided by Frobenius extensions. Recall that an extension $S\subseteq R$ of rings is called a \emph{Frobenius extension} if $R$ is finitely generated and projective
as a left $S$-module and $R\cong\Hom_{S}(_{S}R,S)$ as an $R$-$S$-bimodule.  The invariant properties of rings under Frobenius functors
have been studied by many authors, see \cite{Chen2013,HLGZ,Nakayama,Pappacena,Renwei2018,Renwei,xicc,Zhaozb} for instance.

In this paper, we shall establish relations between Gorenstein projective precovers (in abelian categories with enough projective objects) linked by Frobenius functors, including Frobenius extensions of rings.

To state our main result more precisely, let us first introduce some definitions.

Assume that $F:\mathcal{C}\rightarrow \mathcal{D}$  and $G:\mathcal{D}\rightarrow \mathcal{C}$ are covariant functors between abelian categories. Recall from \cite{Hilton} that if $(F, G)$ is an adjoint pair, then the unit $\eta: 1_{\mathcal{C}}\rightarrow GF$ and the counit $\varepsilon: FG\rightarrow 1_{\mathcal{D}}$ of the adjunction satisfy the identities $\varepsilon_{F(X)} F(\eta_X)=1_{F(X)}$ and $G(\varepsilon_{Y}) \eta_{G(Y)}=1_{G(Y)}$ for all $X\in\mathcal{C}$ and $Y\in\mathcal{D}$. By Lemma \ref{lem0}, for any adjoint pair $(F, G)$, one has $F$ is separable if and only if $\eta: 1_{\mathcal{C}}\rightarrow GF$  is a split monomorphism and $G$ is separable if and only if $\varepsilon:FG\rightarrow 1_{\mathcal{D}}$  is a split epimorphism.
Moreover, if $(F,G)$ and $(G,F)$ are adjoint pairs, then we say that $F$ and $G$ are \emph{Frobenius
functors} and $(F,G)$ is a \emph{Frobenius pair} by \cite{casta}.

Let $\mathcal{A}$ be an abelian category with enough projective objects. Recall that  an object $M$ in $\mathcal{A}$ is called \emph{Gorenstein projective} \cite{EJGP,Sather} if  there exists an exact complexes of projective objects in $\mathcal{A}$:
 $$\mathbf{P}:\cdots\ra P_1\ra P_0\ra P^0\ra P^1\ra \cdots,$$
with $M\cong\textrm{ker}(P^{0}\ra P^{1})$ such that ${\rm Hom}_{\mathcal{A}}(\mathbf{P},Q)$ is exact for any projective object $Q$. In what follows, we denote by $\mathcal{GP}(\mathcal{A})$ the subcategory of $\mathcal{A}$ consisting of Gorenstien projective objects.

Let $\mathcal{X}$ be a class of objects in an abelian category $\mathcal{A}$.  A homomorphism $\varphi:
X\rightarrow M$ with $X\in \mathcal{X}$ is called an
\emph{$\mathcal{X}$-precover} of $M$ \cite{EEEI} if for any
homomorphism $f:X'\rightarrow M$ with $X'\in \mathcal{X}$, there is
a homomorphism $g:X'\rightarrow X$ such that $\varphi g=f$. The class $\mathcal{X}$ is
called \emph{precovering} in $\mathcal{A}$ if every object in $\mathcal{A}$ has an
$\mathcal{X}$-precover.

Now, our main result can be stated as follows.
\begin{thm}\label{tem:GP-precovering} Let $F:\mathcal{C}\rightarrow \mathcal{D}$  and $G:\mathcal{D}\rightarrow \mathcal{C}$ are covariant functors between abelian categories with enough projective objects, and let $(F,G)$ be a Frobenius pair.
 \begin{enumerate}
\item Assume that $G$ is a separable Frobenius functor. If $\GPC$ is precovering in $\mathcal{C}$, then $\GPD$ is precovering in $\mathcal{D}$.

\item Assume that $F$ is a separable Frobenius functor. If $\GPD$ is precovering in $\mathcal{D}$, then $\GPC$ is precovering in $\mathcal{C}$.
\end{enumerate}
\end{thm}

We will apply Theorem \ref{tem:GP-precovering} to Frobenius extensions and excellent extensions. As a result, we produce some examples of rings such that the class of Gorenstein projective modules is precovering over them (see Example \ref{ex:excellent}).

The proof of the above results will be carried out in the next section.

\section{\bf Main results}\label{pre}

We begin this section with the following definition.
\begin{definition}{\rm\label{df:2.1}\cite{NA} A covariant functor $F:\mathcal{C}\rightarrow\mathcal{D}$ is said to be \emph{separable}
if for all objects $M,N$ in $\mathcal{C}$ there are maps $\varphi:\Hom_{\mathcal{D}}(F(M),F(N))\to\Hom_{\mathcal{C}}(M,N)$, satisfying the following conditions:
\begin{enumerate}
 \item For $\theta\in{\Hom_{\mathcal{C}}(M,N)}$ we have $\varphi(F(\theta))=\theta$.
 \item Given $M',N'$ in $\C$, $\alpha\in{\Hom_{\mathcal{C}}(M,M')}$, $\beta\in{\Hom_{\mathcal{C}}(N,N')}$, $f\in{\Hom_{\mathcal{D}}(F(M),F(N))}$, $g\in{\Hom_{\mathcal{D}}(F(M'),F(N'))}$ such that the following diagram is commutative:
     $$\xymatrix{F(M)\ar[d]^{F(\alpha)}\ar[r]^{f}&F(N)\ar[d]^{F(\beta)}\\
F(M')\ar[r]^{g}&F(N'),}$$
then the following diagram is also commutative:
$$\xymatrix{M\ar[d]^{\alpha}\ar[r]^{\varphi(f)}&N\ar[d]^{\beta}\\
M'\ar[r]^{\varphi(g)}&N'.}$$
\end{enumerate}}
\end{definition}

The following lemma collects some results on adjoint functors (see \cite[Theorem 1, p.89]{Maclane} and \cite[Theorem 1.2]{Rafael}):

\begin{lem}\label{lem0} Let $F:\mathcal{C}\rightarrow \mathcal{D}$  and $G:\mathcal{D}\rightarrow \mathcal{C}$ be functors between abelian categories, and let $(F, G)$ be an adjoint pair.
\begin{enumerate}
 \item $F$ is faithful if and only if $\eta_{X}:X\rightarrow GF(X)$  is a monomorphism for any $X\in\mathcal{C}$.
 \item $G$ is faithful if and only if $\varepsilon_{Y}:FG(Y)\rightarrow Y$  is an epimorphism for any $Y\in\mathcal{D}$.
 \item  $F$ is separable if and only if $\eta: 1_{\mathcal{C}}\rightarrow GF$  is a split monomorphism, i.e. there exists a
natural transformation $\psi: GF \rightarrow 1_{\mathcal{C}}$ such that $\psi\eta=1$.
 \item $G$ is separable if and only if $\varepsilon:FG\rightarrow 1_{\mathcal{D}}$  is a split epimorphism, i.e. there exists a
natural transformation $\xi: 1_{\mathcal{D}} \rightarrow FG$ such that $\varepsilon\xi=1$.
 \end{enumerate}
\end{lem}

The following fact collects some results on Frobenius functors (see \cite[Proposition 1.3]{casta}):

\begin{fact} \label{fact:2.1} Let $\mathcal{C}$ and $\mathcal{D}$ are abelian categories and $(F,G)$  a Frobenius pair.
\begin{enumerate}
\item $F$ and $G$ are exact functors.
\item If $\mathcal{C}$ and $\mathcal{D}$ have projective objects, then both $F$ and $G$ preserve projective objects. Moreover, we have $\Ext_{\mathcal{C}}^{i}(X,G(Y)\cong\Ext_{\mathcal{D}}^{i}(F(X),Y)$ and $\Ext_{\mathcal{C}}^{i}(G(Y),X)\cong\Ext_{\mathcal{D}}^{i}(Y,F(X))$
 for all $i\geqslant0$,  $X\in\mathcal{C}$ and  $Y\in\mathcal{D}$.
 \end{enumerate}
\end{fact}

{In what follows, we always assume $(F,G)$ is a Frobenius pair, where $F:\mathcal{C}\rightarrow \mathcal{D}$  and $G:\mathcal{D}\rightarrow \mathcal{C}$ are covariant functors between abelian categories with enough projective objects. }

\begin{prop}\label{thm:3.1}  Let $X$ be an object in $\C$ and $Y$ an object in $\D$.
\begin{enumerate}
\item If $X\in{\GPC}$, then $F(X)\in{\GPD}$. The converse is true if $F$ is faithful.
\item  If $Y\in{\GPD}$, then $G(Y)\in{\GPC}$. The converse is true if $G$ is faithful.
\end{enumerate}
\end{prop}

\begin{proof} We only need to prove (1), and the proof of (2) is similar. The proof is model that of  \cite[Theorem 3.3]{HLGZ}. Let $X$ be a Gorenstein projective object in $\mathcal{C}$. Then there exists an exact complex of projective objects in $\mathcal{C}$:
$$\mathbf{P}:\cdots\ra P_1\ra P_0\ra P^0\ra P^1\ra \cdots$$  with $X\cong \textrm{ker}(P^{0}\ra P^{1})$ such that $\Hom_{\C}(\mathbf{P}, Q)$ is exact for every projective object $Q$ in $\mathcal{C}$. Applying  the functor $F$ to the exact sequence $\mathbf{P}$, we have the following exact sequence of projective objects in $\mathcal{D}$: $$F(\mathbf{P}):\cdots\ra F(P_1)\ra F(P_0)\ra F(P^0)\ra F(P^1)\ra \cdots.$$
Let $Q$ be a projective object in $\mathcal{D}$. It follows that $G(Q)$ is projective in $\mathcal{C}$. Hence the complex $\Hom_{\mathcal{C}}(\mathbf{P}, G(Q))$ is exact, and the complex $\Hom_{\mathcal{D}}(F(\mathbf{P}), Q)$ is exact by adjoint isomorphism. So  $F(X)$ is  Gorenstein projective in $\mathcal{D}$.

 Conversely, we assume that $F(X)$ is a Gorenstein projective in $\mathcal{D}$ and $F$ is faithful. Let $P$ be a projective object in $\mathcal{C}$.  Note that $\Ext_{\mathcal{C}}^{i}(X, GF(P)))\cong\Ext_{\mathcal{D}}^{i}(F(X),F(P))$ for all $i\geqslant1$. Note that $F(P)$ is projective in $\mathcal{D}$ by Proposition \ref{thm:3.1}. Then $\Ext_{\mathcal{D}}^{i}(F(X),F(P))=0$. By Lemma \ref{lem0}, the counit $GF(P)\to P$ is an epimorphism and then $P$ is a direct summand of $GF(P)$. Hence $\Ext_{\mathcal{C}}^{i}(X,P)=0$ for all $i\geqslant1$. It suffices to construct the right part of the complete projective resolution of $X$.  It is easy to check that $GF(X)$ is a Gorenstein projective object in $\mathcal{C}$ by a similar proof above, there exists an exact sequence $0\rightarrow GF(X)\rightarrow P^0\rightarrow L^1\rightarrow0$ in $\mathcal{C}$ with $P^0$ projective and $L^1$ Gorenstein projective. Note that there exists an exact sequence $0\longrightarrow X\s{\eta_{X}}\longrightarrow GF(X)\longrightarrow K\rightarrow 0$  in $\mathcal{C}$ by Lemma \ref{lem0}.
Consider the following pushout diagram:

$$\xymatrix{&&0\ar[d]&0\ar[d]&\\
0\ar[r]&X\ar[r]^{\eta_{X}}\ar@{=}[d]&GF(X)\ar[r]\ar[d]&K
\ar[r]\ar[d]& 0\\
0\ar[r]&X\ar[r]&P^{0}\ar[d]\ar[r]&H^{1}
\ar[r]\ar[d]& 0\\
&&L^{1}\ar@{=}[r]\ar[d]&L^{1}\ar[d]&\\
&&0&\ 0.&\\
}$$
Applying the functor $F$ to the commutative diagram above, we have the following commutative diagram:
$$\xymatrix{&&0\ar[d]&0\ar[d]&\\
0\ar[r]&F(X)\ar[r]^{F(\eta_{X})}\ar@{=}[d]&FGF(X)\ar[r]\ar[d]&F(K)
\ar[r]\ar[d]& 0\\
0\ar[r]&F(X)\ar[r]&F(P^{0})\ar[d]\ar[r]&F(H^{1})
\ar[r]\ar[d]& 0\\
&&F(L^{1})\ar@{=}[r]\ar[d]&F(L^{1})\ar[d]&\\
&&0&\ 0.&\\
}$$
Note that $F(\eta_X)$ is a split monomorphism, thus $F(K)$ is Gorenstein projective because the class of Gorenstein projective objects is closed under direct summands by \cite[Proposition 4.11]{Sather}, and so is  $F(H^{1})$  because the class of Gorenstein projective objects is closed under extensions by \cite[Corollary 4.5]{Sather}. Hence we have an exact sequence $0\rightarrow X\rightarrow P^0\rightarrow H^1\rightarrow 0$ in $\mathcal{C}$ where $P^0$ is projective in $\mathcal{C}$ and $F(H^1)$ is Gorenstein projective in $\mathcal{D}$. By the forgoing proof, we can get that $\Ext_{\mathcal{C}}^{i}(H^{1}, Q)=0$ for all $i\geqslant1$ and any projective object $Q$. Proceed in this manner, we have an exact sequence $0\rightarrow X\rightarrow P^{0}\rightarrow P^{1}\rightarrow \cdots$ in $\mathcal{C}$ with each
$P^i$ projective, which is $\Hom_{\mathcal{C}}(-,Q)$-exact for all projective objects $Q$.
So $X$ is Gorenstein projective.
\end{proof}

Let $X$ be an object in $\mathcal{C}$. The \emph{Gorenstein projective dimension}, ${\rm Gpd}(X)$, of $X$ is defined by declaring that ${\rm Gpd}(X)\leqslant n$ if, and only if there is an exact sequence $0\ra G_{n}\ra \cdots \ra G_{0}\ra X\ra 0$ with all $G_{i}$ Gorenstein projective.

\begin{cor}\label{pro1} If $F$ is faithful, then ${\rm Gpd}(X)= {\rm Gpd}(F(X))$ for any  $X\in{\mathcal{C}}$.
\end{cor}
\begin{proof}
Let $X$ be an object in $\mathcal{C}$.
It is easy to check that ${\rm Gpd}(F(X))\leqslant{\rm Gpd}(X)$, now we will show ${\rm Gpd}(X)\leqslant{\rm Gpd}(F(X))$.
If ${\rm Gpd}(F(X))=\infty$, the equality is trivial. Now assume that ${\rm Gpd}(F(X))=m<\infty$. Consider the following exact sequence in $\mathcal{C}$:
$$0\rightarrow K\rightarrow G_{m-1}\rightarrow \cdots \rightarrow G_{1}\rightarrow G_{0}\rightarrow X\rightarrow 0,$$
where $G_{i}$ is Gorenstein projective in $\C$ for $0\leqslant i \leqslant m-1$. By Proposition \ref{thm:3.1}, we have the following exact sequence in $\mathcal{D}$:

$$0\rightarrow F(K)\rightarrow F(G_{m-1})\rightarrow \cdots \rightarrow F(G_{1})\rightarrow F(G_{0})\rightarrow F(X)\rightarrow 0,$$
where $F(G_{i})$ is Gorenstein projective for $0\leqslant i \leqslant m-1$. It follows from ${\rm Gpd}(F(X))=m$ that $F(K)$ is a Gorenstien projective. Thus $K$ is Gorenstein projective in $\C$ by Proposition \ref{thm:3.1}, and hence ${\rm Gpd}(X)\leqslant m$. So ${\rm Gpd}(X)= {\rm Gpd}(F(X))$. This completes the proof.
\end{proof}
\begin{rem}It should be noted that Corollary \ref{pro1} above is obtained by Chen and Ren in a recent preprint \cite[Theorem 3.2]{Chenren} via a slightly different proof.
\end{rem}

\begin{lem}\label{lem:a-GP-precover} Let $M$ be an object in $\C$ and $N$ an object in $\D$.
\begin{enumerate}
\item If $f:X\rightarrow M$ is a Gorenstein projective precover of $M$, then $F(f):F(X)\to F(M)$ is a Gorenstein projective precover of $F(M)$.

\item If $g:Y\rightarrow N$ is a Gorenstein projective precover of $N$, then $G(g):G(Y)\to G(N)$ is a Gorenstein projective precover of $G(N)$.
\end{enumerate}
\end{lem}

\begin{proof} We only prove (1), and the proof of (2) is similar. Assume that $f:X\rightarrow M$ is a Gorenstein projective precover of $M$. Let $L$ be an object in $\GPD$ and $h:L\to F(M)$ a morphism in $\D$. Then we have the following commutative diagram:
$$\xymatrix@=3.8em{ \Hom_{\D}(L,F(X))\ar[d]^{\cong} \ar[r]^{\tiny\Hom_{\D}(L,F(f))} & \Hom_{\D}(L,F(M)) \ar[d]^{\cong}\\
\Hom_{\C}(G(L),X) \ar[r]^{\tiny\Hom_{\C}(G(L),f)} & \Hom_{\C}(G(L),M).}$$
Since $L\in{\GPD}$, $G(L)\in{\GPC}$ by Proposition \ref{thm:3.1}. Note that $f:X\rightarrow M$ is a Gorenstein projective precover. Then $\Hom_{\C}(G(L),f):\Hom_{\C}(G(L),X)\to \Hom_{\C}(G(L),M)$ is epic. Hence $\Hom_{\D}(L,F(f)):\Hom_{\D}(L,F(X)) \to \Hom_{\D}(L,F(M))$ is epic. This completes the proof.
\end{proof}
By \cite[Proposition 1.2]{Auslander}, one can get that $F(\GPC)$ is precovering in $\GPD$. Moreover, we have the following result.
\begin{lem}\label{lem:GP-precovering-class}
\begin{enumerate}
\item If $\GPC$ is precovering in $\mathcal{C}$, then $F(\GPC)$ is precovering in $\mathcal{D}$.

\item If $\GPD$ is precovering in $\mathcal{D}$, then $G(\GPD)$ is precovering in $\mathcal{C}$.
\end{enumerate}
\end{lem}
\begin{proof} We will prove (1), and the proof of (2) is similar. Let $N$ be an object in $\D$.  Since $\GPC$ is precovering in $\mathcal{C}$ by hypothesis, there exists a Gorenstein projective precover $f:X\ra G(N)$ of $G(N)$. Set $\gamma:=\varepsilon_{N}F(f):F(X)\ra N$, where $\varepsilon_{N}:FG(N)\ra N$ is the counit of the adjoint pair $(F,G)$. Next we claim that $\gamma:F(X)\ra N$ is a $F(\GPC)$-precover. Let $L$ be an object in $\GPC$ and $h:F(L)\to N$ a morphism in $\D$. Then we have the following commutative diagram:
$$\xymatrix@=4.5em{ \Hom_{\D}(F(L),F(X))\ar[d]^{\cong} \ar[r]^{\tiny\Hom_{\D}(F(L),F(f))} & \Hom_{\D}(F(L),FG(N)) \ar[d]^{\cong}\ar[r]\ar[r]^{\tiny\Hom_{\D}(F(L),\varepsilon_{N})}& \Hom_{\D}(F(L),N) \ar[d]^{\cong}\\
\Hom_{\C}(L,GF(X)) \ar[r]\ar[r]^{\tiny\Hom_{\C}(L,GF(f))} & \Hom_{\C}(L,GFG(N))\ar[r]^{\tiny\Hom_{\C}(L,G(\varepsilon_{N}))}& \Hom_{\C}(L,G(N)).}$$
Note that $G(\varepsilon_{N})\circ \eta_{G(N)}=1_{G(N)}$. Then $\Hom_{\C}(L,G(\varepsilon_{N})):\Hom_{\C}(L,GFG(N))\ra\Hom_{\C}(L,G(N))$ is epic. Hence $\Hom_{\D}(F(L),\varepsilon_{N}):\Hom_{\C}(F(L),FG(N))\ra\Hom_{\C}(F(L),N)$ is epic. Since $f$ is a Gorenstein projective precover of $G(N)$, so is $GF(f):GF(X)\ra GFG(N)$ by Lemma \ref{lem:a-GP-precover}. Thus $\Hom_{\C}(L,GF(f)):\Hom_{\C}(L,GF(X))\ra\Hom_{\C}(L,GFG(N))$ is epic, and hence
$\Hom_{\D}(F(L),F(f)):\Hom_{\D}(F(L),F(X))\ra \Hom_{\D}(F(L),FG(N))$ is epic. So $\Hom_{\D}(F(L),\gamma)= \Hom_{\D}(F(L),\varepsilon_{N})\Hom_{\D}(F(L),F(f))$ is epic. This completes the proof.
\end{proof}

We are now in a position to prove Theorem \ref{tem:GP-precovering}.

{\bf Proof of Theorem \ref{tem:GP-precovering}.} We only prove (1), and the proof of (2) is similar. Let $N$ be an object in $\D$. Then there exists a $F(\GPC)$-precover $\alpha:F(X)\ra N$ with $X\in{\GPC}$ by Lemma \ref{lem:GP-precovering-class}. We will show $\alpha:F(X)\ra N$ is a $\GPD$-precover of $N$. Firstly, we have $F(X)\in\GPD$ by Proposition \ref{thm:3.1}. Now let $g:L\ra N$ be a morphism with $L\in{\GPD}$. Since $G$ is a separable functor, there exists a morphism $h:L\ra FG(L)$ such that $\varepsilon_{L}h=1_{L}$. Since $FG(L)\in{F(\GPC)}$ by Proposition \ref{thm:3.1}, there exits a morphism $\beta:FG(L)\ra F(X)$ such that $\alpha\beta=g\varepsilon_{L}$ by Lemma \ref{lem:GP-precovering-class}. So $\alpha(\beta h)=(\alpha\beta)h=(g\varepsilon_{L})h=g(\varepsilon_{L}h)=g$. This completes the proof.\hfill$\Box$

Recall from \cite[Proposition 1.3]{NA} that a Frobenius extension $S\subseteq R$ is separable if and only if $F:= {_{S}R}\otimes_{R}-:{\RM}\rightarrow \SM$ is a separable functor, where $_{R}\mathcal{M}$ is the category of left $R$-modules and $_{S}\mathcal{M}$ is the category of left $S$-modules. As a consequence of Theorem \ref{tem:GP-precovering}, we have the following corollary.

\begin{cor}\label{cor:separable-Frobenius-extension} Assume that $S\subseteq R$ is a separable Frobenius extension. If $\mathcal{GP}(S)$ is precovering in $_{S}\mathcal{M}$, then $\mathcal{GP}(R)$ is precovering in $_{R}\mathcal{M}$.
\end{cor}

It follows from \cite[Lemma 4.7]{Huang} that an excellent extension $S\subseteq R$ with $S$ a commutative ring is a separable Frobenius extension. By Corollary \ref{cor:separable-Frobenius-extension}, we have the following corollary.

\begin{cor}\label{cor:excellent} Assume that $S\subseteq R$ is an excellent extension with $S$ a commutative ring. If $\mathcal{GP}(S)$ is precovering in $_{S}\mathcal{M}$, then $\mathcal{GP}(R)$ is precovering is $_{R}\mathcal{M}$.
\end{cor}
As a consequence of Corollary \ref{cor:excellent} and \cite[Example 2.2]{Huang}, we have the following example.

\begin{ex}\label{ex:excellent}
\begin{enumerate}
\item Let $S$ be ring and $G$ a finite group. If $|G|^{-1} \in S$, then the skew group ring $R=S\ast G$ is an excellent extension of $S$ by \cite[Example 2.2]{Huang}. It follows from Corollary \ref{cor:excellent} that  $\mathcal{GP}(R)$ is precovering in $_{R}\mathcal{M}$ whenever $S$ is a commutative ring such that $\mathcal{GP}(S)$ is precovering in $_{S}\mathcal{M}$.

\item Let $S$ be a finite-dimensional commutative algebra over a field $K$, and let $S'$ be a finite separable field extension
of $K$. Then $R=S\otimes_{K}S'$ is an excellent extension of $S$. Hence $\mathcal{GP}(R)$ is precovering in $_{S}\mathcal{M}$ by  Corollary \ref{cor:excellent}.

\end{enumerate}
\end{ex}
\bigskip \centerline {\bf ACKNOWLEDGEMENTS}
\bigskip
The authors would like to thank Professor Xiaowu Chen for kindly pointing out that Corollary \ref{pro1} above is obtained in a recent preprint \cite{Chenren} via a
slightly different proof. We
are very grateful to him for the reference and helpful comments. The authors would like to thank
 Wei Ren and Yongliang Sun for helpful discussions on parts of this article.
\bigskip

\renewcommand\refname
\vspace{4mm}
\small

\noindent\textbf{Jiangsheng Hu}\\
School of Mathematics and Physics, Jiangsu University of Technology,
 Changzhou 213001, China\\
E-mail: jiangshenghu@jsut.edu.cn\\[1mm]
\textbf{Huanhuan Li}\\
School of Mathematics and Statistics, Xidian University,
 Xi'an 710071, China\\
 lihuanhuan0416@163.com\\[1mm]
 \textbf{Jiafeng L${\rm \ddot{u}}$}\\
 Department of Mathematics, Zhejiang Normal University,
 Jinhua 321004, China\\
 jiafenglv@zjnu.edu.cn\\[1mm]
\textbf{Dongdong Zhang}\\
Department of Mathematics, Zhejiang Normal University,
 Jinhua 321004, China\\
E-mail: zdd@zjnu.cn\\[1mm]

{\bf References}
\begin{thebibliography}{99}

\bibitem{Asadollahi}J. Asadollahi, T. Dehghanpour and R. Hafezi, {\it On the existence of Gorenstein projective precovers}, Rend. Semin. Mat. Univ. Padova {\bf 136} (2016), 257-264.

\bibitem{Auslander} M. Auslander, I. Reiten,
Homologically finite subcategories. Representations of algebras and related topics (Kyoto, 1990),
1-42, London Math. Soc. Lecture Note Ser., 168, Cambridge Univ. Press, Cambridge, 1992.



%\bibitem{Bell1993}Bell, A., Farnsteiner, R. (1993). On the theory of Frobenius extensions and its applications to Lie
%superalgebras. Trans. Amer. Math. Soc.  335: 407-424.
%\bibitem{DB} Bennis, D.,  Mahdou, N. (2010). Global Gorenstein dimensions. Proc. Amer. Math.
%Soc.  138: 461-465.
%\bibitem{Bennis} Bennis, D. (2009). Rings over which the class of Gorenstein
%at modules is closed under extensions. Comm. Algebra 37: 855-868.
\bibitem{BHG} D. Bravo, M. Hovey and J. Gillespie, {\it The stable module category of a general ring}, preprint, arXiv:1405.5768.
\bibitem{casta} F. Casta\~{n}o Iglesias, J. G\'{o}mez Torrecillas and C. Nastasescu, {\it Frobenius functors:
applications}, Comm. Algebra  {\bf 27} (1999), 4879-4900.

%\bibitem{DNDV} D$\check{a}$sc$\check{a}$lescu, S., N$\check{a}$st$\check{a}$sescu, C., Del Rio, A., Van Oystaeyen, F. (1996).  Gradings of finite support. Application to injective objects. J. Pure Appl. Algebra. 107: 193-206.


\bibitem{Chen2013} X. W. Chen, {\it Totally reflexive extensions and modules}, J. Algebra {\bf 379} (2013), 322-332.

\bibitem{Chenren} X. W. Chen, W. Ren, {\it Totally reflexive extensions and modules}, arXiv:2008.11467v2, 2020.

%\bibitem{CET} L.W. Christensen, S. Estrada, P. Thompson, Homotopy categories of totally acyclic complexes with applications to the flat cotorsion theory, Preprint arXiv:1812.04402.
%\bibitem{I on the fin}  Emmanouil, I. (2012). On the finiteness of Gorenstein homological dimensions. J. Algebra 372: 376-396.
\bibitem{EEEI} E. E. Enochs, {\it Injective and flat covers, envelopes and resolvents},
Israel J. Math. {\bf 39} (1981), 189-209.

\bibitem{EEI} E. E. Enochs, S. Estrada and A. Iacob, {\it Gorenstein injective, projective and flat (pre)covers}, Acta Math. Univ. Comenian. (N.S.) {\bf 83} (2014), 217-230.


\bibitem{EJGP} E. E. Enochs and O. M. G. Jenda,  {\it Gorenstein injective and projective modules}, Math. Z. {\bf 220} (1995), 611-633.

 %\bibitem{EJ}  Enochs, E. E.,  Jenda, O. M. G. (2000). Relative Homological Algebra. Berlin, New York: Walter de Gruyter.

%  \bibitem{EJ GF} Enochs, E. E., Jenda, O. M. G., Torrecillas, B. (1993). Gorenstein flat  modules. Nanjing daxue xuebao
% shuxue bannian kan 10: 1-9.

\bibitem{EIO} S. Estrada, A. Iacob and S. Odabasi, {\it Gorenstein projective and flat (pre)covers},  Publ. Math. Debrecen {\bf 91} (2017), 111-121.
\bibitem{EIK} S. Estrada, A. Iacob and K. Yeomans,  {\it Gorenstein Projective Precovers}, Mediterr. J. Math. {\bf 14} (2017), Paper No. 33, 10pp.

%\bibitem{Fischman} Fischman, D., Montgomery, S., Schneider, H. J. (1997). Frobenius extensions of subalgebras of Hopf
%algebras. Trans. Amer. Math. Soc. 349: 4857-4895.

\bibitem{Gillespie} J. Gillespie and M. Hovey,  {\it Gorenstein model structures and generalized derived categories}, Proc. Edinb. Math. Soc. {\bf 53} (2010), 675-696.


%\bibitem{Kadison} Kadison, L. (1999). New Examples of Frobenius Extensions. Univ. Lecture. Ser., vol. 14.
% Providence, R. I.: Amer. Math. Soc.
%
%\bibitem{Kadison2} Kadison, L. (1999). Separability and the twisted Frobenius bimodule. Algebr. Represent. Theory 2: 397-414.

\bibitem{Hilton} P. J. Hilton and U. Stammbach, {\it A course in homological algebra}, Graduate Texts in Mathematics 4, Springer-
Verlag, 1971.

\bibitem{HLGZ} J. S. Hu, H. H. Li, Y. X. Geng and D. D. Zhang,  {\it Frobenius
functors and Gorenstein flat dimensions}, Comm. Algebra {\bf 48} (2020), 1257-1265.

\bibitem{Huang} Z. Y. Huang and J. X. Sun,  {\it Invariant properties of represenations under excellent extensions}, J.
Algebra {\bf 358} (2012), 87-101.



%\bibitem{Ho} Holm, H. (2004). Gorenstein homological dimensions. J. Pure Appl. Algebra 189: 167-193.
%\bibitem{Holm}  Holm, H. (2004). Gorenstein Homological Algebra. Ph.D. thesis, Copenhagen: University of
%Copenhagen.
\bibitem{Jorgensen} P. J$\o$rgensen, {\it Existence of Gorenstein projective resolutions and Tate cohomology},
J. Eur. Math. Soc {\bf 9} (2007), 59-76.
\bibitem{Maclane} S. Mac Lane, {\it Categories for the working mathematician}, Second Edition, Graduate Texts in Mathematics
5, Springer-Verlag, 1998.
\bibitem{Morita} K. Morita, {\it Adjoint pairs of functors and Frobenius extensions}, Sci. Rep. Tokyo Kyoiku
Daigaku Sect. A {\bf 9} (1965), 40-71.

\bibitem{NA} C. N$\breve{a}$st$\check{a}$sescu, M. Van den Bergh and F. Van Oystaeyen, {\it Separable functors
applied to graded rings}, J. Algebra {\bf 123} (1989), 397-413.
\bibitem{Nakayama} T. Nakayama and T. Tsuzuku, {\it On Frobenius extensions I}, Nagoya Math. J. {\bf 17} (1960), 89-110.
\bibitem{Pappacena} C. J. Pappacena, {\it Frobenius bimodules between noncommutative
spaces}, J. Algebra {\bf 275} (2004), 675-731.

\bibitem{Rafael} M. D. Rafael, {\it Separable functors revisited}, Comm. Algebra {\bf 18} (1990), 1445-1459.

\bibitem{Renwei2018} W. Ren, {\it Gorenstein projective modules and Frobenius extensions}, Sci. China Math. {\bf 61} (2018), 1175-1186.
\bibitem{Renwei}  W. Ren, {\it Gorenstein projective and injective dimensions over
Frobenius extensions}, Comm. Algebra {\bf 46} (2018), 5348-5354.

\bibitem{Sather} S. Sather-Wagstaff, T. Sharif and D.  White, {\it Stability of Gorenstein categories}, J. Lond. Math. Soc. {\bf 77} (2008), 481-502.

%\bibitem{JJarxiv} $\check{\rm{S}}$aroch J,  $\check{\rm{S}}\check{\rm{t}}$ov\'{\i}$\check{\rm{c}}$ek J. \emph{Singular compactness and definability for $\Sigma$-cotorsion and Gorenstein modules}. arXiv:1804.09080.
%
%\bibitem{Schneider}  Schneider, H. J. (1992). Normal basis and transitivity of crossed products for Hopf algebras. J. Algebra 151: 289-312.

\bibitem{xicc} C. C. Xi, {\it Frobenius bimodules and flat-dominant dimensions}, Sci. China Math. {\bf 63} (2020). https://doi.org/ 10.1007/s11425-018-9519-2.

\bibitem{Zhaozb}  Z. B. Zhao, {\it Gorenstein homological invariant properties under Frobenius extensions}, Sci. China Math. {\bf 62} (2019), 2487-2496.



\end{thebibliography}
\end{document}